\documentclass[12pt]{article}

\usepackage[english]{babel}
\usepackage[cp1251]{inputenc}
\usepackage[T2A]{fontenc}

\usepackage{amssymb}
\usepackage{amsmath}
\usepackage{amsthm} 
\usepackage{mathtext}
\usepackage{amsfonts}

\renewcommand{\le}{\leqslant}
\renewcommand{\ge}{\geqslant}
\renewcommand{\leq}{\leqslant}
\renewcommand{\geq}{\geqslant}

\newcommand{\lda}{\lambda}

\newtheorem{theorem}{Theorem}
\newtheorem{lemma}{Lemma}

\newtheorem*{mainlemma}{The Main Lemma}

\theoremstyle{definition}
\newtheorem{definition}{Definition}
\newtheorem{remark}{Remark}

\begin{document} 

\title{On the asymptotics of integrals related to the generalized Cantor ladder}
\author{A.I.~Nazarov\footnote{St.Petersburg State University, Faculty of Mathematics and Mechanics.}, 
N.V.~Rastegaev\footnote{St.Petersburg State University, Faculty of Mathematics and Mechanics,
P.L. Chebyshev laboratory.}}
\maketitle

\abstract{The Cantor ladder is naturally included into various families
of self-similar functions. In the frame of these families
we study the asymptotics of some parametric integrals.}

\section{Introduction}
Let $\{I_k = [a_k,b_k]\}_{k=1}^m$ be subsegments of $[0,1]$ with non-intersecting interiors.
Denote by $S_k(t) = a_k + (b_k-a_k)\,t$ the affine contractions of $[0,1]$ onto $I_k$ preserving the
orientation. We also introduce a set of positive numbers
 $\{\rho_k\}_{k=1}^m$ such that \mbox{$\sum\limits_{k=1}^m\rho_k = 1$}. 

Define the operator $\mathcal{S}$ acting in the space $L_{\infty}(0,1)$ by the formula
\[\mathcal{S}(f) = \sum_{k=1}^m\left(\chi_{I_k}(f\circ S_k^{-1})+\chi_{\{x>b_k\}}\right)\rho_k. \]
It is easy to check, see, e.g., \cite{Sh}, that $\mathcal{S}$ is a contracting map in $L_{\infty}(0,1)$.
Thus, there exists a unique function $C\in L_{\infty}(0,1)$ such that $\mathcal{S}(C)=C$.

We call such a function $C(t)$ \textit{the generalized Cantor ladder} with $m$ steps. 
It can be found as a uniform limit of the sequence $\mathcal{S}^k(f)$ with $f(t)\equiv t$. 
This allows to assume $C(t)$ continuous and monotone with $C(0)=0$, $C(1)=1$. 

Note that the derivative of $C(t)$ in the sense of distributions is a measure $\mu$ self-similar in
the sense of Hutchinson (see \cite{H}). This means
\[\mu(E) = \sum_{k=0}^m \rho_k \mu(S_k^{-1}(E\cap I_k)).\]
More general self-similar functions are described in \cite{Sh}.

For a generalized Cantor ladder $C(t)$ we study the asymptotic behavior, as $\lda\to\infty$, of the 
integral
\[ E(\lda) = \int\limits_0^1e^{\lda C(t)} \; dt.\]

\begin{remark}
It is easy to see that the quest of asymptotics of $E(\lda)$ as $\lda\to-\infty$ can be reduced to a
similar problem as $\lda\to+\infty$.

Namely, let a ladder $C(t)$ be generated by segments $I_k = \left[a_k, b_k\right]$, $k=1,\dots,m$,
and by numbers $\{\rho_k\}_{k=1}^m$. Consider the ladder $C_1(t)$ generated by segments 
$\{J_k\}_{k=1}^m$ with $J_{m-k+1} = \left[1-b_k, 1-a_k\right]$ and numbers $\sigma_k = 1-\rho_{m-k+1}$,
$k=1,\dots,m$. For these ladders we have an obvious relation
\[ E_C(-\lda) = e^{-\lda} E_{C_1}(\lda). \]
Thus the quest of asymptotics of $E_C(\lda)$ as $\lda\to-\infty$ can be reduced to the quest of
asymptotics of $E_{C_1}(\lda)$ as $\lda\to+\infty$. In what follows we assume $\lambda>0$.
\end{remark}

\begin{definition}
We say that a generalized Cantor ladder is \textit{regular} if 
\[ \forall k=2,\dots, m \qquad \rho_k =\rho_1 = \frac{1}{m}, \quad 
b_k - a_k=b_1 - a_1, \quad a_k - b_{k-1}=a_2 - b_1. \]
For $a_2 = b_1$ such a ladder degenerates to $C(t)\equiv t$, and we have $E(\lda) = \frac{e^\lda-1}{\lda}$.
\end{definition}

The regular ladder for $m=2$ was considered in the paper \cite{GK}. In particular, the first term of the
asymptotic series for $E(\lda)$ was calculated. We also mention the paper \cite{G} where the function
$E(\lda)$ and some other integrals were expressed (in the case of classical Cantor ladder) in terms of
series of elementary functions.

\section{The recurrent relation and the Main Lemma}

Without loss of generality we can assume $a_1 = 0$, $b_m = 1$ (any another case can be reduced to this
one by dilation). Denote by $\Delta_i$, $i = 1\ldots 2m-1$, the lengths of parts of the segment $[0,1]$, 
i.e. $\Delta_{2k-1} = b_k - a_k >0$, $\Delta_{2k} = a_{k+1} - b_k\ge0$. We define also 
$h_k = \sum\limits_{i=1}^k \rho_i$, $g_k = 1 - h_k$.

\begin{remark}
The relation $\mathcal{S}(C)=C$ can be rewritten as follows:
\begin{equation*}
C(t) =
\left\{
\begin{array}{ll}
h_{k-1} + \rho_k C\left(\frac{t-a_k}{\Delta_{2k-1}}\right), & t \in [a_k, b_k]\\
h_k, & t \in [b_k, a_{k+1}]\\
\end{array} 
\right.
\end{equation*} 
\end{remark}

\begin{lemma}\label{lemma_main_eq}
For a ladder with $m$ steps the following relation holds:
\begin{equation}\label{main_eq}
E(\lda) = \Delta_1 E(\rho_1\lda) + \Delta_2 e^{h_1\lda} + 
\ldots + \Delta_{2m-1}e^{h_{m-1}\lda}E(\rho_m\lda).
\end{equation}
\end{lemma}
\begin{proof}
\begin{align*}
E(\lda) &= \int\limits_0^1e^{\lda C(t)} \; dt =
\sum\limits_{k=1}^m \int\limits_{a_k}^{b_k} e^{\lda C(t)} \; dt +
\sum\limits_{k=1}^{m-1} \int\limits_{b_k}^{a_{k+1}} e^{\lda C(t)} \; dt =\\
&= \sum\limits_{k=1}^m e^{h_{k-1}\lda}\int\limits_{a_k}^{b_k} e^{\rho_k\lda C\left(\frac{t-a_k}{\Delta_{2k-1}}\right)} \; dt +
\sum\limits_{k=1}^{m-1} \Delta_{2k} e^{h_k\lda} = \\
&= \sum\limits_{k=1}^m \Delta_{2k-1} e^{h_{k-1}\lda} E(\rho_k\lda) +
\sum\limits_{k=1}^{m-1} \Delta_{2k} e^{h_k\lda},
\end{align*}
and we arrive at \eqref{main_eq}.
\end{proof}

To analyse this relation we need the following statement.

\begin{mainlemma}\label{main_lemma}
Let the function $F(\lda)$, $\lda \geq 0$, satisfy the following conditions:
\begin{enumerate}
\item $1\leqslant F(\lda) \leq e^\lda$;
\item $F(\eta\lda) = d \, e^{(\eta-1)\lda} F(\lda) + f(\lda), \quad 0 < d < 1, \; \eta > 1$;
\item $f(\lda) = O(e^{(\eta-\varepsilon)\lda})$ as $\lda \to \infty$, for some $\varepsilon > 0$.
\end{enumerate}
Then, as $\lda \to \infty$, the asymptotic relation
\[ F(\lda) = \Phi(\log_\eta(\lda)) \lda^\alpha e^\lda + O(e^{(1-\varepsilon)\lda})\]
holds with $\alpha = \log_\eta(d)<0$ and $1$-periodic function $\Phi$.
\end{mainlemma}
\begin{remark}
In a particular case this statement was proved in \cite{GK}.
\end{remark}
\begin{proof}
We introduce the notation
\[ F_1(\lda) = \dfrac{F(\lda)}{\lda^{\alpha}e^{\lda}},\quad f_1(\lda) = \dfrac{f(\lda)}{d\lda^{\alpha}e^{\eta\lda}}. \]
Then the assumption 2 can be rewritten as follows:
\[ F_1(\eta\lda) = F_1(\lda) + f_1(\lda). \]
By induction we obtain
\[ F_1(\lda) = F_1\left( \dfrac{\lda}{\eta^{N}} \right) + \sum\limits_{k=1}^N f_1\left(\dfrac{\lda}{\eta^k}\right). \]
Note that $F_1(\lda) \leq \lda^{-\alpha} \to 0$ as $\lda \to 0$. Whence we can write
\[ F_1(\lda) = \sum\limits_{k=1}^{\infty} f_1\left(\dfrac{\lda}{\eta^k}\right). \]
Now we introduce the functions
\[ G(\lda) = \sum\limits_{k=0}^{\infty} f_1(\eta^k\lda),\quad H(\lda) = F_1(\lda) + G(\lda). \]

The estimate $f_1(\lda) = O(\lda^{-\alpha}e^{-\varepsilon\lambda})$ implies that $G(\lda)$ is well 
defined, and $G(\lda) = O(\lda^{-\alpha}e^{-\varepsilon\lambda})$. Further, by construction we have
$H(\eta\lda) = H(\lda)$, i.e. $H(\lda)$ is a $1$-periodic function of $\log_\eta(\lda)$.

We denote $\Phi(x) = H(\eta^x)$ and conclude that $F_1(\lda) = \Phi(\log_\eta(\lda)) + 
O(\lda^{-\alpha}e^{-\varepsilon\lambda})$. Then we turn back to the function $F(\lda)$, and the
statement follows.
\end{proof}

\section{The asymptotics of $E(\lda)$}
\subsection{The first term}
We claim that, for any generalized Cantor ladder, the function $E(\lda)$ satisfies the assumptions 
of the Main Lemma. Indeed, $0 \leq C(t) \leq 1$ implies 
$1 \leq E(\lda) \leq e^\lda$ for all $\lda\geq0$.
Further, define $\eta = \dfrac{1}{\rho_m} > 1$. Then we can rewrite the relation \eqref{main_eq} 
as follows:
\[ E(\eta\lda) = \Delta_{2m-1} e^{(\eta-1)\lda} E(\lda) + f(\lda), \]
where $f(\lda) = \Delta_1 E\left(\rho_1 \eta\lda\right) + 
\Delta_2 e^{h_1 \eta\lda} + \ldots + \Delta_{2m-2} e^{(\eta-1)\lda} = O(e^{(\eta-1)\lda})$.

Applying the Main Lemma we obtain
\begin{equation}\label{eq_E_Phi} 
E(\lambda) = H(\lda)\lambda^\alpha e^\lambda + O(1),
\end{equation}
where $\alpha = \log_\eta(\Delta_{2m-1})<0$, $H(\lda) = \Phi(\log_\eta(\lambda))$, $\Phi(x)$ is 
$1$-periodic.

The function $H(\lambda)$ is a sum of series which converges uniformly on any compact in the
half-plane $Re(\lambda) > 0$. Therefore, analyticity of $f(\lambda)$ implies analyticity of $\Phi(x)$ 
in the strip $\lvert Im(x) \rvert < \frac{\pi}{2\ln(\eta)}$. 
In general case it is difficult to say anything more since $f(\lambda)$ is expressed in terms of 
$E(\lambda)$. For example, in a degenerate case $\rho_k = \Delta_{2k-1}$, $\Delta_{2k}=0$ we obtain 
$C(t)\equiv t$, $E(\lda) = \frac{e^\lda-1}{\lda}$, and thus $\Phi(x)$ becomes a constant.
In general case even the question whether $\Phi(x)$ is constant remains open. However, for regular
ladders the dependence of $f(\lambda)$ on $E(\lambda)$ can be eliminated. Then $\Phi(x)$ 
can be written in a more explicit form. This allows us to obtain additional information.

For a (non-degenerate) regular ladder we have $\eta=m$, $\alpha<-1$, and \eqref{main_eq} can be 
rewritten as follows:
\[
E(m\lambda) = \Delta_1 \dfrac{e^{m\lambda}-1}{e^{\lambda}-1}\,E(\lambda) + 
\Delta_2 \dfrac{e^{(m-1)\lambda}-1}{e^{\lambda}-1}\, e^{\lambda}.
\]
We introduce the functions
\[
\tilde{F_1}(\lambda) = \dfrac{E(\lambda)}{\lambda^\alpha \left( e^{\lambda} - 1\right)}, \qquad
\tilde{f_1}(\lambda) = \dfrac{\Delta_2\left(e^{(m-1)\lambda}-1\right)e^{\lambda}}
{\Delta_1 (e^{\lambda}-1)(e^{m\lambda}-1) \lambda^\alpha}
\]
and obtain
\[
\tilde{F_1}(m\lambda) = \tilde{F_1}(\lambda) + \tilde{f_1}(\lambda).
\]
Repeating the proof of the Main Lemma we arrive at
\begin{equation}\label{eq_E_Phi_tilde}
E(\lambda) = \tilde{\Phi}(\log_m(\lambda))\lambda^\alpha (e^\lambda - 1) + O(1) = 
\tilde{\Phi}(\log_m(\lambda))\lambda^\alpha e^\lambda + O(1),
\end{equation}
where $\tilde{\Phi}(x) = \sum\limits_{k\in\mathbb{Z}} \tilde{f_1}(m^{k+x})$ is $1$-periodic function. 
From relations \eqref{eq_E_Phi} and \eqref{eq_E_Phi_tilde} we conclude that 
$\Phi(\log_m(\lambda))-\tilde{\Phi}(\log_m(\lambda)) = O(\lambda^{-\alpha}e^{-\lambda})$, i.e. 
$\Phi(x) \equiv \tilde{\Phi}(x)$.

Thus, we have the explicite formula for $\Phi(x)$. Now we can study the Fourier series
\[
\Phi(x) = \sum_{n\in \mathbb{Z}} c_n e^{2\pi i n x} .
\]
To proceed we need the Riemann formula, see, e.g., \cite{BE}:
\[
\zeta(\lambda) = \dfrac{1}{\Gamma(\lambda)}\int_0^\infty \dfrac{t^{\lambda-1}}{e^t-1}\, dt.
\]

\begin{theorem}
For a regular ladder, the Fourier coefficients of the function $\Phi(x)$ can be evaluated as follows
\begin{equation}\label{eq_Phi_Fourier} 
c_n = \dfrac{\Delta_2(1-\Delta_1)}{\Delta_1 \ln(m)} \Gamma(\alpha_n) \zeta(\alpha_n),
\end{equation}
where $\alpha_n = -\alpha-\dfrac{2\pi i n}{\ln(m)}$. 
\end{theorem}

\begin{remark}
Since $Re\,\alpha_n > 1$, this implies $c_n\neq 0$ if a regular ladder is non-degenerate 
($\Delta_2 \neq 0$). In particular, in this case $\Phi(x) \neq const$.
For $m=2$ the formula \eqref{eq_Phi_Fourier} was obtained in \cite{GK}.
\end{remark}

\noindent {\it Proof}. We have
\begin{align*}
c_n &= \int\limits_0^1 \Phi(s) e^{-2\pi i n s} \, ds =
\sum_{j\in \mathbb{Z}}\int_0^1\tilde{f_1}(m^{s+j})e^{-2\pi i n s}\,ds = \\
    &= \dfrac{1}{\ln(m)} \sum_{j\in \mathbb{Z}}\int_{m^j}^{m^{j+1}}\tilde{f_1}(t)e^{-2\pi i n \log_m(t)}\,\frac{dt}{t} =
      \dfrac{1}{\ln(m)} \int_0^\infty t^\alpha \tilde{f_1}(t)t^{\alpha_n-1}\,dt = \\
    &= \dfrac{\Delta_2}{\Delta_1 \ln(m)} \int_0^\infty \dfrac{\left(e^{(m-1)t}-1\right)e^{t}}
 {(e^{t}-1)(e^{m t}-1)} \; t^{\alpha_n-1}\,dt = \\
&= \dfrac{\Delta_2}{\Delta_1 \ln(m)} \left( \int_0^\infty \dfrac{t^{\alpha_n-1}}{e^{t}-1} \, dt - 
\int_0^\infty \dfrac{t^{\alpha_n-1}}{e^{m t}-1} \, dt \right) = \\
&= \dfrac{\Delta_2}{\Delta_1 \ln(m)} \Gamma(\alpha_n) \zeta(\alpha_n) \left( 1 - m^{-\alpha_n} \right) =
  \dfrac{\Delta_2(1-\Delta_1)}{\Delta_1 \ln(m)} \Gamma(\alpha_n) \zeta(\alpha_n).\tag*{$\square$}
\end{align*}

\subsection{More terms in the simplest case}
Let us continue to study the asymptotic expansion. We begin from the simple example.

\begin{theorem}
Let $m=2$, $\rho_1 = \rho_2=\frac{1}{2}$. Then the function $E(\lda)$ can be represented as follows:
\begin{equation}\label{E_as_row}
E(\lda) = H(\lda)\lda^\alpha e^\lda + \sum\limits_{k=0}^{\infty}e^{-k\lda}\left(C_{k} + D_{k} H(\lda)\lda^\alpha\right),
\end{equation}
where the series converges uniformly for sufficiently large $\lda$.

Here $C_k$, $D_k$ are numbers satisfying the following recurrent relations:
\begin{equation}\label{rekur_2}
\begin{split}
&C_0 = -\frac{\Delta_2}{\Delta_3},\qquad D_0 = -\frac{\Delta_1}{\Delta_3},\\
C_{k+1} =&
\left\{
\begin{array}{ll}
-\dfrac{\Delta_1}{\Delta_3}C_{k}, & k \equiv 1 \pmod{2} \\
\dfrac{1}{\Delta_3}C_{k/2}-\dfrac{\Delta_1}{\Delta_3}C_{k},  & k \equiv 0 \pmod{2} \\
\end{array} 
\right.\\
D_{k+1} =&
\left\{
\begin{array}{ll}
-\dfrac{\Delta_1}{\Delta_3}D_{k}, & k \equiv 1 \pmod{2} \\
D_{k/2}-\dfrac{\Delta_1}{\Delta_3}D_{k},   & k \equiv 0 \pmod{2}. \\
\end{array} 
\right.\\
\end{split}
\end{equation}
\end{theorem} 

\begin{proof}
The relation \eqref{main_eq} in this case can be rewritten as follows:
\begin{equation}\label{eq_E1_m1} 
E(2\lda) = E(\lda)(\Delta_1 + \Delta_3 e^\lda) + \Delta_2 e^\lda.
\end{equation}
Applying the Main Lemma we can write the result as follows:
\[ E(\lda) = H(\lda)\lda^\alpha e^\lda(1+E_1(\lda)),\quad E_1(\lda) = O(\lda^{-\alpha}e^{-\lda}). \]
We substitute this into \eqref{eq_E1_m1} and obtain
\begin{equation}\label{eq_E1_m2} 
E_1(2\lda) = E_1(\lda)\left(1 + \frac{\Delta_1}{\Delta_3} e^{-\lda}\right) + 
\left(\frac{\Delta_2}{\Delta_3}\frac{\lda^{-\alpha}}{H(\lda)} e^{-\lda} + \frac{\Delta_1}{\Delta_3} e^{-\lda}\right).
\end{equation}
This implies
\[ E_1(\lda) + \frac{\Delta_2}{\Delta_3}\frac{\lda^{-\alpha}}{H(\lda)} e^{-\lda} = 
E_1(2\lda) - E_1(\lda)\frac{\Delta_1}{\Delta_3} e^{-\lda} - 
\frac{\Delta_1}{\Delta_3} e^{-\lda} = O(e^{-\lda}).\]
Denote by $E_2(\lambda)$ the right-hand side of the last equality. Then
\[ E_1(\lda) = -\frac{\Delta_2}{\Delta_3}\frac{\lda^{-\alpha}}{H(\lda)} e^{-\lda} + E_2(\lda),\quad E_2(\lda) = O(e^{-\lda}). \]
This gives us the second term of the asymptotics
\[ E(\lda) = H(\lda)\lda^\alpha e^\lda - \frac{\Delta_2}{\Delta_3} + O(\lda^{\alpha}). \]
We can substitute it into the relation \eqref{eq_E1_m1} and obtain the expression for $E_2(\lda)$ 
similar to \eqref{eq_E1_m2}:
\[ E_2(2\lda) = E_2(\lda)\left(1 + \frac{\Delta_1}{\Delta_3} e^{-\lda}\right) + 
\left(\frac{\Delta_1}{\Delta_3} e^{-\lda} + \frac{\Delta_2(1 - \Delta_1)}{\Delta_3^2}\frac{\lda^{-\alpha}}{H(\lda)} e^{-2\lda}\right).\]
Repeating this algorithm we obtain formulas \eqref{rekur_2} and \eqref{E_as_row} as asymptotic expansion.
Next, from \eqref{rekur_2} we conclude that coefficients $C_k$, $D_k$ grow not faster then an exponent
of their number:
\[ \lvert C_k \rvert \leq \lvert C_0 \rvert \left(\dfrac{2}{\Delta_3}\right)^k, \quad
\lvert D_k \rvert \leq \lvert D_0 \rvert \left(\dfrac{2}{\Delta_3}\right)^k. \]
This gives us the uniform convergence of the series in the right-hand side of \eqref{E_as_row} if $\lambda$ is sufficiently
large.

It remains to show that the right-hand side of \eqref{E_as_row} exhausts $E(\lda)$.
To do this, consider the remainder
\[ {\mathfrak E}(\lda) = E(\lda) - H(\lda)\lda^\alpha e^{\lda} - 
\sum\limits_{k=0}^{\infty}e^{-k\lda}\left(C_{k} + D_{k} H(\lda)\lda^\alpha\right). \]
Note that the sequence $E_k(\lda)$ converges to
к ${\mathfrak E}_1(\lda):=\frac{e^{-\lda}}{\lda^{\alpha}H(\lda)}{\mathfrak E}(\lda)$ in the space
$L_\infty(\Lambda,+\infty)$ for sufficiently large $\Lambda$. 
Further,
\begin{multline*}
 \left\lvert E_k(2\lda)-E_k(\lda)\left(1 + \frac{\Delta_1}{\Delta_3} e^{-\lda}\right)\right\rvert 
\leq\\
\leq \sum\limits_{j=k}^{\infty} \left( \lvert C_0 \rvert \left(\dfrac{2}{\Delta_3}\right)^{j-1}
\frac{1}{\lda^\alpha H(\lda)} + 
\lvert D_0 \rvert \left(\dfrac{2}{\Delta_3}\right)^{j-1}\right) e^{-j\lda}
\end{multline*}
tends to zero in $L_\infty(\Lambda,+\infty)$. Therefore, ${\mathfrak E}_1(\lda)$ satisfies the homogeneous
equation 
\begin{equation}\label{F1_homogen}
{\mathfrak E}_1(2\lda) = {\mathfrak E}_1(\lda)\left(1 + \frac{\Delta_1}{\Delta_3} e^{-\lda}\right), \quad \lda>\Lambda.
\end{equation}
We know that for any $\varsigma\geq 1$ the estimate ${\mathfrak E}_1(\lda) = O(e^{-\varsigma\lda})$ holds.
Whence for some $c>0$, $\varsigma \geq 1$ we have
\begin{equation}\label{F1_estim}
\lvert {\mathfrak E}_1(\lda) \rvert \leq c\, e^{-\varsigma\lda} \quad \mbox{for} \; \lda > \Lambda.
\end{equation}
From \eqref{F1_homogen} and \eqref{F1_estim} we conclude
\begin{multline*} 
\lvert {\mathfrak E}_1(\lda) \rvert = 
\lvert {\mathfrak E}_1(2\lda) - \frac{\Delta_1}{\Delta_3} e^{-\lda}{\mathfrak E}_1(\lda) \rvert \leq {} \\
{} \leq c\, e^{-2\varsigma\lda} + \frac{\Delta_1}{\Delta_3}\, c\, e^{-(\varsigma+1)\lda} \leq  
 \frac{1}{\Delta_3}\, c\, e^{-(\varsigma+1)\lda} \leq  c\, e^{-(\varsigma+1+\frac{\ln(\Delta_3)}{\Lambda})\lda}. 
\end{multline*}
Without loss of generality we can assume $\Lambda > -2\ln(\Delta_3)$. Then
 \[ \lvert {\mathfrak E}_1(\lda) \rvert \leq c\, e^{-(\varsigma+\frac{1}{2})\lda} \quad \mbox{for} \; \lda > \Lambda. \]
Repeating this argument we obtain the relation \eqref{F1_estim} with the same constant $c$ and arbitrary
$\varsigma \geq 1$. Thus, ${\mathfrak E}_1(\lda)\equiv 0$ for all $\lda > \Lambda$, which completes the
proof.
\end{proof}

\begin{remark}
For $\Delta_1 = \Delta_3$, i.e. for a regular ladder, \eqref{rekur_2} implies $D_k = 0$ for all $k \geq 1$. 
This fact is true in general case, see Theorem \ref{regular} below.
\end{remark}

\subsection{More terms in the case $\rho_m =\min\{\rho_i\}$}

In this subsection we transfer our scheme to a general case. 
Unfortunately, it is not always possible. Here we introduce an additional assumption:
$\rho_m =\min\{\rho_i\}$.
We rewrite the statement of the Main Lemma as follows:
\[
E(\lambda) = e^{\lambda} \left( H_1(\lambda) + E_1(\lambda) \right), \qquad\qquad
H_1(\lambda) =  H(\lambda)\lambda^\alpha, \quad E_1(\lambda) = O(e^{-\lambda}).
\]
We substitute this into \eqref{lemma_main_eq} and rewrite the obtained equation as follows:
\begin{equation}\label{main_eq_E_1}
\frac 1{\Delta_{2m-1}} E_1(\eta\lda) = E_1(\lda)+
\sum\limits_{i=1}^{m-1}\frac{\Delta_{2i-1}}{\Delta_{2m-1}}\,e^{-g_i\eta\lda}E_1(\rho_i\eta \lda) - \mathfrak{P}_1(\lda),
\end{equation}
\begin{equation*}
\mathfrak{P}_1(\lda) = \sum\limits_{\varsigma \in \mathfrak{I}_1} c_\varsigma^1(\lda)e^{-\varsigma\lda}.
\end{equation*}
Here 
\[ 
\mathfrak{I}_1 = \left\{\eta g_k\right\}_{k=1}^{m-1}, \quad 
c_{\eta g_k}^1(\lambda) = -\frac{\Delta_{2k-1} H_1(\eta\rho_k\lambda)+\Delta_{2k}}{\Delta_{2m-1}}. 
\]
Note that the minimal element in $\mathfrak{I}_1$ is $\eta g_{m-1} = 1$. We transform \eqref{main_eq_E_1} 
as follows:
\begin{multline}\label{eq_E1a}
E_2(\lda) := E_1(\lda) - c_{1}^1(\lda)e^{-\lda}=\\
=\frac 1{\Delta_{2m-1}}\,E_1(\eta\lda)-\sum\limits_{i=1}^{m-1}\frac{\Delta_{2i-1}}{\Delta_{2m-1}}\,e^{-g_i\eta\lda}E_1(\rho_i\eta \lda)+
\Big[\mathfrak{P}_1(\eta\lda)-c_{1}^1(\lda)e^{-\lda}\Big].
\end{multline}
We know that $E_1(\lambda) = O(e^{-\lambda})$. Therefore all terms in the right-hand side of 
\eqref{eq_E1a} are $O(e^{-\varsigma'\lda})$,  $\varsigma' > 1$, whence 
$E_2(\lda) =O(e^{-\varsigma'\lda})$. Thus,
\[E_1(\lda)=-\left(\frac{\Delta_{2m-2}}{\Delta_{2m-1}}+\frac{\Delta_{2m-3}}{\Delta_{2m-1}}H_1(\eta\rho_{m-1}\lda)\right)e^{-\lda}+O(e^{-\varsigma'\lda}).
\]
Now we can rewrite \eqref{main_eq_E_1} as follows:
\begin{equation*}
\frac 1{\Delta_{2m-1}} E_2(\eta\lda) = E_2(\lda)+
\sum\limits_{i=1}^{m-1}\frac{\Delta_{2i-1}}{\Delta_{2m-1}}\,e^{-g_i\eta\lda}E_2(\rho_i\eta \lda) - \mathfrak{P}_2(\lda),
\end{equation*}
\begin{equation*}
\mathfrak{P}_2(\lda) = \sum\limits_{\varsigma \in \mathfrak{I}_2} c_\varsigma^2(\lda)e^{-\varsigma\lda},\qquad
\mathfrak{I}_2 \subseteq (\mathfrak{I}_1 \setminus \{1\}) \cup \{\eta, \eta (\rho_i +g_i) \}.
\end{equation*}
Note that even for $\varsigma \in \mathfrak{I}_1 \cap \mathfrak{I}_2$ the coefficients 
$c_\varsigma^2(\lda)$ in general differ from $c_\varsigma^1(\lda)$. However, this relation is quite
similar to \eqref{main_eq_E_1}. Therefore, we can hope that this algorithm can be iterated.

Let us write down a general form of the iteration. We have a function $E_k(\lda)$ satisfying the 
following relations:
\begin{equation}\label{main_eq_E_k}
\frac 1{\Delta_{2m-1}} E_k(\eta\lda) = E_k(\lda)+
\sum\limits_{i=1}^{m-1}\frac{\Delta_{2i-1}}{\Delta_{2m-1}}\,e^{-g_i\eta\lda}E_k(\rho_i\eta \lda) - \mathfrak{P}_k(\lda),
\end{equation}
\begin{equation*}
\mathfrak{P}_k(\lda) = \sum\limits_{\varsigma \in \mathfrak{I}_k} c_\varsigma^k(\lda)e^{-\varsigma\lda}.
\end{equation*}
\begin{equation*}
E_k(\lda) = O(e^{-\varsigma_k \lda}), \quad 
\varsigma_k \leq \varsigma_k': = \min\limits_{\varsigma \in \mathfrak{I}_k}\varsigma.
\end{equation*}
We rewrite \eqref{main_eq_E_k} as follows:
\begin{multline*}
E_k(\lda)-c_{\varsigma_k'}^k(\lda)e^{-\varsigma_k'\lda}=\\
=\frac 1{\Delta_{2m-1}}E_k(\eta\lda)-\sum\limits_{i=1}^{m-1}\frac{\Delta_{2i-1}}{\Delta_{2m-1}}\,e^{-g_i\eta\lda}E_k(\rho_i\eta \lda)+
\Big[\mathfrak{P}_k(\lda)-c_{\varsigma_k'}^k(\lda)e^{-\varsigma_k'\lda}\Big].
\end{multline*}
Note that
\[
E_k(\eta\lda) = O(e^{-\eta\varsigma_k\lda}), \quad \mbox{and}\quad \eta\varsigma_k > \varsigma_k;
\]
\[
e^{-g_i\eta\lda}E_k(\rho_i\eta\lda) = O(e^{-\eta(g_i+\rho_i\varsigma_k)\lda}), \quad \mbox{and} \quad
\eta(g_i+\rho_i\varsigma_k) > \eta\varsigma_k \rho_i \geq \varsigma_k,
\]
in the last inequality we use the assumption $\rho_m =\min\{\rho_i\}$;
\[
\mathfrak{P}_k(\lda) - c_{\varsigma_k'}^k(\lda)e^{-\varsigma_k'\lda} = 
O(e^{-\varsigma_k''\lda}), \quad \mbox{and} \quad
\varsigma_k'': = \min\limits_{\varsigma \in \mathfrak{I}_k \setminus \{\varsigma_k'\}}\varsigma 
> \varsigma_k' \geq \varsigma_k.
\]
This implies
\[
E_{k+1}(\lda) := E_k(\lda)-c_{\varsigma_k'}^k(\lda)e^{-\varsigma_k'\lda} 
= O(e^{-\varsigma_{k+1}\lda}), \quad \varsigma_{k+1} > \varsigma_k.
\]
After substitution we obtain for $E_{k+1}(\lda)$ a relation similar to \eqref{main_eq_E_k}.
It remains to make sure that $\varsigma_{k+1} \leq \varsigma_{k+1}'$:
\[
\mathfrak{I}_{k+1} \subseteq (\mathfrak{I}_k \setminus \{\varsigma_k'\}) \cup 
\{\eta\varsigma_k', \eta(\rho_i\varsigma_k' + g_i) \};
\]
\begin{multline*}
\varsigma_{k+1} = \min\{\eta\varsigma_k, \eta(g_i+\rho_i\varsigma_k), \varsigma_k''\} \leq\\
\leq\min\left(\{\eta\varsigma_k', \eta(g_i+\rho_i\varsigma_k')\} \cup 
(\mathfrak{I}_k \setminus \{\varsigma_k'\})\right) \leq
\min\limits_{\varsigma \in \mathfrak{I}_{k+1}}\varsigma = \varsigma_{k+1}'.
\end{multline*}
Thus, we can separate more and more new terms.

\begin{theorem}
Let $\rho_m =\min\{\rho_i\}$. Then the function $E(\lda)$ can be represented as a series
\begin{multline}\label{E_as_row_2}
 E(\lda) = H(\lda)\lda^\alpha e^\lambda-
\left(\frac{\Delta_{2m-2}}{\Delta_{2m-1}}+\Delta_{2m-3}H\Big(\frac{\rho_{m-1}}{\rho_m}\lda\Big)
\left(\rho_{m-1}\lda\right)^{\alpha}\right)+\\
+\sum_{\varsigma\in \mathfrak{I}} c_\varsigma(\lda)e^{(1-\varsigma)\lda}
\end{multline}
(all exponents in the last sum are negative). This series converges uniformly for sufficiently large $\lda$.
\end{theorem}

\begin{proof}
The calculations above give us \eqref{E_as_row_2} as asymptotic expansion.
For $c_\varsigma(\lda)$, as for coefficients $C_k$, $D_k$ in the simplest case, we have a recurrence:
\[ c_\varsigma(\lda) = c_\varsigma^1(\lda)+\frac 1{\Delta_{2m-1}}c_{\varsigma/\eta}(\eta\lda)-
\sum_{i=1}^{m-1} \frac{\Delta_{2i-1}}{\Delta_{2m-1}}\,c_{(-g_i+\varsigma/\eta)/\rho_i}(\rho_i\eta\lda). \]
To prove the convergence of the series \eqref{E_as_row_2}, one should show that the exponents
$\varsigma$ grow sufficiently fast while coefficients $c_\varsigma(\lda)$ grow sufficiently slowly.

First we show by induction that there exist $C_1 > 0$, $C_2 > 1$, such that
\begin{equation}\label{c_bounded}
 \lvert c_\varsigma(\lda) \rvert \leq C_1 C_2^\varsigma.
\end{equation}
Note that for any $C_2 > 1$ there exists $C_1^{(0)}$ such that the estimate \eqref{c_bounded} holds for 
$c_\varsigma^1(\lda)$. Next, let \eqref{c_bounded} be satisfied for some first terms in the series
\eqref{E_as_row_2}. We claim that \eqref{c_bounded} holds for the next term. Indeed,
\begin{multline*}
\lvert c_\varsigma(\lda) \rvert \leq C_1^{(0)}C_2^\varsigma +
\frac {C_1}{\Delta_{2m-1}}\,C_2^{\varsigma/\eta} + \\
+\sum_{i=1}^{m-1}C_1\frac{\Delta_{2i-1}}{\Delta_{2m-1}}C_2^{\varsigma-\frac{g_i}{\rho_i}} \leq
C_1 C_2^\varsigma \left(\frac{C_1^{(0)}}{C_1} + \frac 2{\Delta_{2m-1}}\,C_2^{-\varepsilon}\right),
\end{multline*}
where $\varepsilon = \min\{ \frac {\eta-1}{\eta}\min\limits_{\varsigma\in\mathfrak{I}}\varsigma, 
\min\limits_{i < m}\frac{g_i}{\rho_i}\}$. Setting
$C_2 = \left(\frac{1}{4}\Delta_{2m-1}\right)^{-\frac 1{\varepsilon}}$ and $C_1 = 2C_1^{(0)}$, we obtain
 \eqref{c_bounded}.

Now we study the exponents in $\mathfrak{P}_k$. We introduce linear functions
\[ l_0(\varsigma) = \rho_m\varsigma, \quad l_i(\varsigma) = g_i + \rho_i\varsigma, \; i=1,\ldots, m-1,
\quad l_m(\varsigma)=\varsigma. \]
Any step of the algorithm can be described as follows: we take away the term with minimal exponent 
$\varsigma$ from $\mathfrak{P}_k$ and add this term to the series \eqref{E_as_row_2}. In this process
some terms with exponents $l_0^{-1}(l_i(\varsigma))$, $i=1,\ldots, m$ are added or changed in
$\mathfrak{P}_{k+1}$. 

The assumption $\rho_m =\min\{\rho_i\}$ implies that the graph of $l_0(\varsigma)$ does not intersect
graphs of other $l_i$ for $\varsigma > 0$. Therefore, the linear transforms
$l_0^{-1}(l_i(\varsigma))$, $i=1,\ldots, m$, have no positive fixed points. Thus, the sequence of 
exponents has no concentration points. This is shown at the Figure~\ref{img1} which shows the graphs
of $l_i(\varsigma)$ for the regular ladder with $m=2$. 

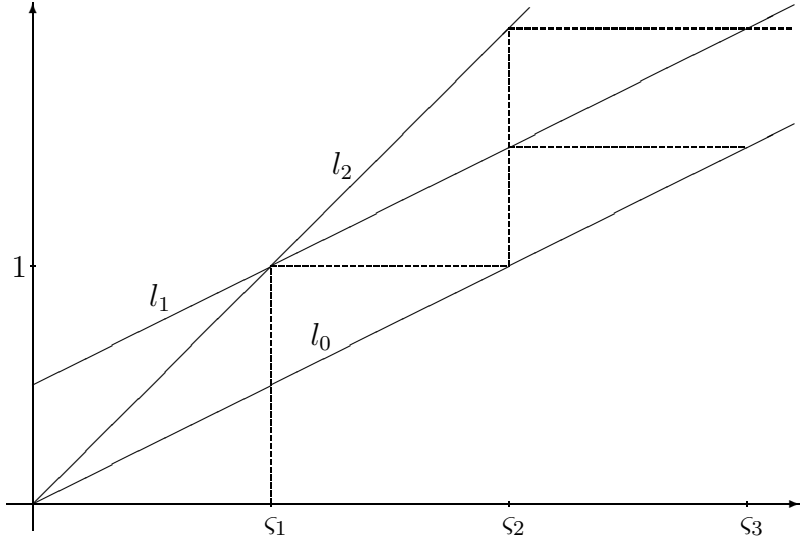
\begin{figure}[h]
\begin{center}
\begin{picture}(300,200)
\put(0,10){\vector(1,0){300}}%
\put(10,0){\vector(0,1){200}}%
\put(10,10){\line(1,1){188}}%
\put(123,135){$l_2$}%
\put(10,10){\line(2,1){288}}%
\put(115,70){$l_0$}%
\put(10,55){\line(2,1){288}}%
\put(54,85){$l_1$}%
\put(100,9){\line(0,1){2}}%
\put(97.5,0){$\varsigma_1$}%
\put(190,9){\line(0,1){2}}%
\put(187.5,0){$\varsigma_2$}%
\put(280,9){\line(0,1){2}}%
\put(277.5,0){$\varsigma_3$}%
%
\put(9,100){\line(1,0){2}}%
\put(2,96){$1$}%
%
\multiput(100,10)(0,3){30}{\line(0,1){2}}%
\multiput(100,100)(3,0){30}{\line(1,0){2}}%
\multiput(190,100)(0,3){30}{\line(0,1){2}}%
\multiput(190,145)(3,0){30}{\line(1,0){2}}%
\multiput(190,190)(3,0){36}{\line(1,0){2}}%
\end{picture}
\caption{The sequence of exponents $\varsigma_k$ for a regular ladder}
\label{img1}
\end{center}
\end{figure}

So, instead of the term with exponent $\varsigma$ any step of the algorithm adds to $\mathfrak{P}_k$ 
at most $m$ other terms with exponents greater than $\varsigma+\delta$ with some $\delta>0$. To estimate
the series in \eqref{E_as_row_2} we change all new exponents to the minimal one (note that all the
exponents arising at subsequent steps also decrease). Taking \eqref{c_bounded} into account we obtain
for  $\lda > \ln (C_2)$
\begin{align}\label{estim_row}
\sum_{\varsigma\in \mathfrak{I}} |c_\varsigma(\lda)|e^{(1-\varsigma)\lda} &\leq 
\sum_{\varsigma\in \mathfrak{I}} C_1 C_2^\varsigma\ e^{(1-\varsigma)\lda} =
C_1 C_2\sum_{\varsigma\in \mathfrak{I}}e^{(1-\varsigma)(\lda-\ln C_2)} \leq\nonumber\\
&\leq C_1 C_2\sum_{\varsigma\in \mathfrak{I}_0} \sum_{k=0}^{\infty} m^k e^{(1-(\varsigma+k\delta))(\lda-\ln C_2)} =\nonumber\\
&= C_1 C_2\sum_{\varsigma\in \mathfrak{I}_0} 
\left(e^{(1-\varsigma)(\lda-\ln C_2)}\sum_{k=0}^{\infty} e^{-k\delta(\lda-\ln C_2-\frac{\ln(m)}{\delta})}\right).
\end{align}
The last series converges uniformly for sufficiently large $\lambda$.

To complete the proof, as in the simplest case, we consider the remainder
\[ {\mathfrak E}(\lda) = E(\lda) - H(\lda)\lda^\alpha e^\lambda +
\frac{\Delta_{2m-2}}{\Delta_{2m-1}}+\Delta_{2m-3}H\Big(\frac{\rho_{m-1}}{\rho_m}\lda\Big)
\left(\rho_{m-1}\lda\right)^{\alpha}
- \sum_{\varsigma\in \mathfrak{I}} c_\varsigma(\lda)e^{(1-\varsigma)\lda} \]
and note that the sequence $E_k(\lda)$ converges to
${\mathfrak E}_1(\lda):=e^{-\lda}{\mathfrak E}(\lda)$ in the space
$L_\infty(\Lambda,+\infty)$ for sufficiently large $\Lambda$. Further,
\[ \left\lvert\frac 1{\Delta_{2m-1}} E_k(\eta\lda) - E_k(\lda) -
\sum\limits_{i=1}^{m-1}\frac{\Delta_{2i-1}}{\Delta_{2m-1}}\,e^{-g_i\eta\lda}E_k(\rho_i\eta \lda)\right\rvert \leq
 e^{-\lda}\mathfrak{F}_k,\]
where $\mathfrak{F}_k$ are tails of the series \eqref{estim_row}. Since this series converges uniformly
for $\lda > \Lambda$, we conclude that ${\mathfrak E}_1(\lda)$ satisfies the homogeneous equation 
\begin{equation}\label{F1_homogen_2}
\frac 1{\Delta_{2m-1}} \mathfrak{E}_1(\eta\lda) = \mathfrak{E}_1(\lda) +
\sum\limits_{i=1}^{m-1}\frac{\Delta_{2i-1}}{\Delta_{2m-1}}\,e^{-g_i\eta\lda}\mathfrak{E}_1(\rho_i\eta \lda), \quad \lda>\Lambda.
\end{equation}
As in the simplest case, for some $c>0$, $\varsigma \geq 1$ we have
\begin{equation}\label{F1_estim_2}
\lvert {\mathfrak E}_1(\lda) \rvert \leq c\, e^{-\varsigma\lda} \quad \mbox{for} \; \lda > \Lambda.
\end{equation}
From \eqref{F1_homogen_2} and \eqref{F1_estim_2} we obtain
\begin{multline*} 
\lvert {\mathfrak E}_1(\lda) \rvert = 
\left\lvert \frac 1{\Delta_{2m-1}} \mathfrak{E}_1(\eta\lda) -
\sum\limits_{i=1}^{m-1}\frac{\Delta_{2i-1}}{\Delta_{2m-1}}\,e^{-g_i\eta\lda}
\mathfrak{E}_1(\rho_i\eta \lda)\right\rvert\leq {} \\
{} \leq \frac 1{\Delta_{2m-1}}c\, e^{-\eta\varsigma\lda} + 
\sum\limits_{i=1}^{m-1}\frac{\Delta_{2i-1}}{\Delta_{2m-1}}\, 
c\, e^{-\eta(\rho_i\varsigma+g_i)\lda} \leq  \\
\leq \frac{2}{\Delta_{2m-1}}\, c\, e^{-(\varsigma+\delta)\lda} \leq  c\, e^{-\left(\varsigma+\delta-\frac{\ln(2/\Delta_{2m-1})}{\Lambda}\right)\lda}. 
\end{multline*}
Without loss of generality we can assume $\Lambda > \frac 2{\delta}\ln (\frac 2 {\Delta_{2m-1}})$. Then
 \[ \lvert {\mathfrak E}_1(\lda) \rvert \leq c\, e^{-(\varsigma+\frac{\delta}{2})\lda} \quad \mbox{for} \; \lda > \Lambda. \]

As in the simplest case, this gives ${\mathfrak E}_1(\lda)\equiv0$ for $\lda > \Lambda$, and the statement follows.
\end{proof}

\begin{remark}
It is easy to see that if we know the expansion \eqref{E_as_row_2} we can reconstruct the parameters of 
the function $C(t)$.
\end{remark}

Now we consider the case of the regular ladder.

\begin{theorem}\label{regular}
For a regular ladder the relation \eqref{E_as_row_2} is simplified and reads as follows:
\[ E(\lda) = H(\lda)\lda^\alpha e^\lda - (\frac{\Delta_2}{\Delta_1} + H(\lda)\lda^\alpha) + \sum_{k=1}^{+\infty}C_k e^{-k\lda}. \]
\end{theorem}
\begin{proof}
We slightly change the definition of $E_1(\lda)$:
\[ E(\lda) = H_1(\lda)(e^\lda-1)+e^\lda E_1(\lda). \]
Then the relation \eqref{main_eq_E_1} becomes
\begin{equation*}
\frac 1{\Delta_1}E_1(m\lda) = E_1(\lda) + \sum\limits_{j=1}^{m-1}e^{-j\lda}E_1(\lda) - \mathfrak{P}_1(\lda),
\end{equation*}
\begin{equation*}
\mathfrak{P}_1(\lda) = \frac{\Delta_2}{\Delta_1}\sum\limits_{j=1}^{m-1} e^{-j\lda}.
\end{equation*}
The function $H(\lda)$ is absent in this relation. Therefore it cannot arise in subsequent terms of the asymptotics.
\end{proof}

\subsection{The ladders with a critical point}

If the assumption $\rho_m =\min\{\rho_i\}$ is not satisfied we can in general give only asymptotic
expansion for $E(\lda)$.

The assumption $\rho_m =\min\{\rho_i\}$ was used only in the development of the relation
$\eta(g_i+\rho_i\varsigma_k) > \varsigma_k$. In general case this relation becomes the inequality
\[ \varsigma_k < \dfrac{g_i}{\rho_m-\rho_i} \]
for all $i$ such that $\rho_i < \rho_m$. We call the number
$\varsigma^o=\min\limits_{i:\;\rho_i<\rho_m}\dfrac{g_i}{\rho_m-\rho_i}$ \textit{the critical point} 
of generalized Cantor ladder. Note that $\varsigma^o>1$.

It is clear that we can separate new terms until $\varsigma_k < \varsigma^o$, and not all
$c_\varsigma(\lda)$ with $\varsigma_k < \varsigma < \varsigma^o$ vanish. Note that the first condition
is stable: if $\varsigma_k < \varsigma^o$ then
\[ \varsigma_{k+1} = \min\{\eta\varsigma_k, \eta(g_i+\rho_i\varsigma_k), \varsigma_k''\} <
\eta\min\{g_i+\rho_i\varsigma^o\} \le \varsigma^o.
\]

Unfortunately, vanishing of all $c_\varsigma(\lda)$ for $\varsigma_k < \varsigma < \varsigma^o$
is possible though in a somewhat degenerate case. For example, one can consider the classical Cantor 
ladder with two steps of the width $\frac{1}{3}$ but define it in an alternative way. Namely,
consider a ladder with three steps: $I_1=[0,\frac{1}{9}]$, $I_2=[\frac{2}{9},\frac{1}{3}]$, 
$I_3=[\frac{2}{3},1]$, $\rho_1=\rho_2=\frac{1}{4}$, $\rho_3=\frac{1}{2}$. This gives the same classical
ladder with the same asymptotics of $E(\lda)$. But for this definition the ladder has a critical point
 $\varsigma^o=2$. Since this critical point cannot be a concentration point for the exponents, all 
$c_\varsigma(\lda)$ for $\varsigma_k < \varsigma < \varsigma^o$ should vanish for some $k$.

For completeness, we give an example of a ladder with non-vanishing sequence of $c_\varsigma(\lda)$. Let
$\lvert I_1 \rvert = \lvert I_2 \rvert = \Delta<\frac 12$, $\rho_1 < \rho_2 < \frac{1}{\sqrt{2}}$; 
for example, set $\rho_1 = \frac{1}{3}$, $\rho_2 = \frac{2}{3}$. For this ladder the relation 
\eqref{main_eq_E_1} becomes
\[ E_1(\lda) = \Delta e^{-\rho_2\lda}E_1(\rho_1\lda)+\Delta E_1(\rho_2\lda)+c_1(\lda)e^{-\rho_2\lda}, \]
\[ c_1(\lda) = \Delta H_1(\rho_1\lda)+(1-2\Delta). \]
Taking the next term away we arrive at
\[ E_2(\lda) = \Delta e^{-\rho_2\lda}E_2(\rho_1\lda)+\Delta E_2(\rho_2\lda)+c_2(\lda)e^{-\lda}, \]
\[ c_2(\rho_2\lda) = \frac 1{\Delta} c_{1}(\lda) - c_{1}(\rho_1\lda).  \]

If $c_{2}(\lda) = 0$ then $c_1(\lda)$ should have the form
\begin{equation}\label{c1}
c_1(\lda) = \lda^{\alpha'}\Phi'(\log_{\rho_1}(\lda)),
\end{equation}
where $\alpha' = -\log_{\rho_1}(\Delta)$ while $\Phi'$ is a $1$-periodic function. From another side,
\begin{equation}\label{c1'}
c_1(\lda) = \Delta(\rho_1\lda)^{\alpha}\Phi(-\log_{\rho_2}(\rho_1\lda))+(1-2\Delta),
\end{equation}
where $\alpha = -\log_{\rho_2}(\Delta)$ and $\Phi$ is $1$-periodic.
It is easy to see that \eqref{c1} and \eqref{c1'} are asymptotically incompatible, since $1-2\Delta\neq 0$.

In the subsequent steps we have a unique term with exponent less then $\rho_m\varsigma^o$. Corresponding
coefficients satisfy $c_{k+1}(\lda) = -c_k(\rho_1\lda)$. Therefore, $c_\varsigma(\lda)$ cannot vanish
all together, and the asymptotic expansion has infinitely many terms.

This situation is shown at the Figure~\ref{img2}. One can see the intersection of graphs of 
$l_0(\varsigma)$ and $l_1(\varsigma)$ providing the concentration point, the sequence of exponents 
tending to this point, and an exponent greater then $\varsigma^o$, which cannot arise in our asymptotic
expansion.

\begin{figure}[h]
\begin{center}
\begin{picture}(240,200)
\put(0,10){\vector(1,0){240}}%
\put(10,0){\vector(0,1){200}}%
\put(10,10){\line(1,1){188}}%
\put(124,135){$l_2$}%
\put(10,10){\line(3,2){228}}%
\put(115,70){$l_0$}%
\put(10,70){\line(3,1){228}}%
\put(54,90){$l_1$}%
\put(100,9){\line(0,1){2}}%
\put(97.5,0){$\varsigma_1$}%
\put(145,9){\line(0,1){2}}%
\put(142.5,0){$\varsigma_2$}%
\put(168,9){\line(0,1){2}}%
\put(165.5,0){$\varsigma_3$}%
\put(180,9){\line(0,1){2}}%
\put(177.5,0){$\varsigma_4$}%
\put(190,9){\line(0,1){2}}%
\put(187.5,0){$\varsigma^o$}%
\put(9,100){\line(1,0){2}}%
\put(2,96){$1$}%
\multiput(100,10)(0,3){30}{\line(0,1){2}}%
\multiput(100,100)(3,0){15}{\line(1,0){2}}%
\multiput(145,100)(0,3){15}{\line(0,1){2}}%
\multiput(145,145)(3,0){30}{\line(1,0){2}}%
\multiput(145,115)(3,0){8}{\line(1,0){2}}%
\multiput(168,115)(0,3){3}{\line(0,1){2}}%
\multiput(168,123)(3,0){4}{\line(1,0){2}}%
\multiput(180,123)(0,3){1}{\line(0,1){2}}%
\multiput(180,126.5)(3,0){2}{\line(1,0){2}}%
\multiput(190,12)(0,6){20}{\line(0,1){4}}%
\end{picture}
\caption{The sequence of exponents $\varsigma_k$ for a ladder with a critical point}
\label{img2}
\end{center}
\end{figure}
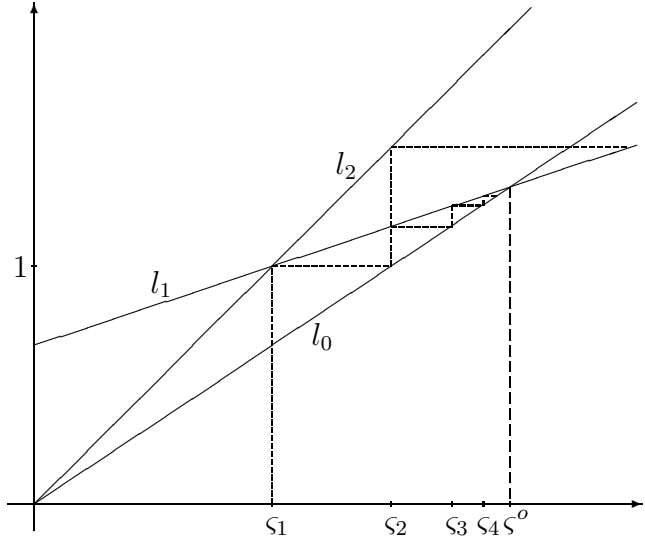

\begin{theorem}
Let a ladder have a critical point. Then the function $E(\lda)$ can be expanded into the asymptotic sum
\begin{multline*}
 E(\lda) = H(\lda)\lda^\alpha e^\lambda-\left(\frac{\Delta_{2m-2}}{\Delta_{2m-1}}
+\Delta_{2m-3}H\left(\frac{\rho_{m-1}}{\rho_m}\lda\right)\left(\rho_{m-1}\lda\right)^{\alpha}\right)+\\
+\sum_{\varsigma\in \mathfrak{I}'} c_\varsigma(\lda)e^{(1-\varsigma)\lda} + O(e^{(1-\varsigma')\lda}),
\end{multline*}
for any given $\varsigma'<\varsigma^o$. All elements of $\mathfrak{I}'$ satisfy the inequality 
$1<\varsigma<\varsigma'$.

If the coefficients $c_\varsigma(\lda)$ for $\varsigma < \varsigma^o$ do not vanish all together, this
sum can have arbitrarily many terms.
\end{theorem}

We are grateful to E.A.~Gorin who pointed out the paper \cite{G}, and to I.A.~Sheipak for some useful
remarks. Our work was supported by RFBR grant 10-01-00154a and by St.Petersburg State University grant 
N6.38.64.2012. The second author was also supported by the Chebyshev Laboratory under the grant
11.G34.31.0026 of the Government of the Russian Federation.


\medskip

\begin{thebibliography}{99}

\bibitem{BE}
  Bateman H., Erd\'elyi~A., {\em Higher transcendental functions}, NY-Toronto-London, 1953.
\bibitem{G}
  Gordon~R.A., {\em Some integrals involving the Cantor function}, Amer. Math. Monthly, {\bf 116} (2009), N3, 218--227. 
\bibitem{GK}
Gorin~E.A., Kukushkin~B.N., {\em Integrals related to the Cantor function}, Algebra \& Analysis,
{\bf 15} (2003), N3, 188--220 (Russian). English transl: St. Petersburg Math. J. {\bf 15} (2004), 449--468. 
\bibitem{H}
  Hutchinson~J.E., {\em Fractals and Self Similarity}, Indiana Univ. Math. Journ., {\bf 30} (1981), N5, 713--747.
\bibitem{Sh}
  Sheipak~I.A., {\em On the construction and some properties of self-similar functions in the spaces $L_p[0,1]$}, 
Matem. zametki, {\bf 81} (2007), N6, 924--938 (Russian). English transl.: Mathematical Notes, {\bf 81} (2007), N6, 
827--839.

\end{thebibliography}
\end{document}